\documentclass[12pt]{amsart}
\usepackage{amssymb,latexsym,amsthm}
\usepackage{amsmath,amsfonts,amssymb}
\usepackage{graphicx}
\usepackage{color}
\usepackage{mathrsfs}

\newcommand{\z}{\mbox{\boldmath $z$}}

\newcommand{\lip}{\operatorname{Lip}}

\newcommand{\ch}{\operatorname{Ch}}

\newcommand{\iso}{\operatorname{Iso}(E_1,E_2)}
\newcommand{\isonly}{\operatorname{Iso}}
\newcommand{\M}{M(E_1,E_2)}
\newcommand{\gp}{G_{\Pi}}
\newcommand{\gpiee}{\gp\cap\iso}
\newcommand{\G}{\gp\cap\iso}
\newcommand{\mc}{\mathbb C}
\newcommand{\tone}{T_1}
\newcommand{\isoa}{\operatorname{Iso}(A_1,A_2)}
\newcommand{\gpa}{G_{\Pi_0}\cap\isoa}
\newcommand{\pep}{p_{\varepsilon}}

\newcommand{\pedonem}{p_{\varepsilon}'\left(\frac{1}{m}\right)}
\newcommand{\qe}{q_{\varepsilon}}

\newcommand{\qedonem}{q_{\varepsilon}'\left(\frac{1}{m}\right)}
\newcommand{\wdonem}{w'\left(\frac{1}{m}-t\right)}
\newcommand{\Imarg}{\operatorname{Im}}
\newcommand{\Real}{\operatorname{Re}}
\newcommand{\mf}{{\mathfrak M}}
\newcommand{\ext}{\operatorname{ext}}
\newcommand{\SSS}{\mathcal{S}}
\newtheorem{theorem}{Theorem}
\newtheorem{lemma}[theorem]{Lemma}
\newtheorem{cor}[theorem]{Corollary}
\newtheorem{prop}[theorem]{Proposition}

\theoremstyle{definition}

\theoremstyle{remark}

\begin{document}
\author{
Osamu~Hatori
}
\address{
Department of Mathematics, Faculty of Science,
Niigata University, Niigata 950-2181, Japan
}
\email{hatori@math.sc.niigata-u.ac.jp
}

\author{
Shiho Oi
}
\address{
Niigata Prefectural Hakkai High School, Minamiuonuma 
949-6681 Japan.
}
\email{shiho.oi.pmfn20@gmail.com
}


\title[]
{2-local isometries on function spaces}




\begin{abstract}
We study 2-local reflexivity of the set of all surjective isometries between certain function spaces. We do not assume linearity for isometries. We prove that a 2-local isometry in the group of all surjective isometries on the algebra of all continuously differentiable functions on the closed unit interval with respect to several norms is a surjective isometry. We also prove that a 2-local isometry in the group of all surjective isometries on the Banach algebra of all Lipschitz functions on the closed unit interval with the sum-norm is a surjective isometry.
\end{abstract}
\maketitle
\section{Introduction}\label{sec1}

 Motivated by the paper by Kowalski and S\l odkowski \cite{ks}, the concept of 2-locality was introduced by \v Semrl,who obtained the first results on 2-local automorphisms and 2-local derivations \cite{smrl}. Moln\'ar \cite{mol2l} studied 2-local isometries on operator algebras. 
Given a metric space $\mathfrak{M}_j$ for $j=1,2$ an isometry from $\mathfrak{M}_1$ into $\mathfrak{M}_2$ is a distance preserving map. The set of all surjective isometries from $\mathfrak{M}_1$ onto $\mathfrak{M}_2$ is denoted by $\operatorname{Iso}(\mathfrak{M}_1,\mathfrak{M}_2)$, and $\operatorname{Iso}(\mathfrak{M})$ if $\mathfrak{M}_1=\mathfrak{M}_2=\mathfrak{M}$. 
We say a map $T:\mathfrak{M}_1\to \mathfrak{M}_2$ is 2-local in $\operatorname{Iso}(\mathfrak{M}_1,\mathfrak{M}_2)$ if for every pair $x,y\in \mathfrak{M}_1$ there exists a surjective isometry $T_{x,y}\in \operatorname{Iso}(\mathfrak{M}_1,\mathfrak{M}_2)$ such that
\[
\text{
$T(x)=T_{x,y}(x)$ and $T(y)=T_{x,y}(y)$.
}
\] 
In this case we say that $T$ is a 2-local isometry. 
It is obvious by the definition that a 2-local isometry is in fact an isometry, which needs not to be surjective. Hence a 2-local isometry $T$ belongs to $\operatorname{Iso}(\mathfrak{M}_1,\mathfrak{M}_2)$ if $T$ is surjective. 
We say that $\operatorname{Iso}(\mathfrak{M}_1,\mathfrak{M}_2)$ is 2-local reflexive if every 2-local isometry belongs to $\operatorname{Iso}(\mathfrak{M}_1,\mathfrak{M}_2)$. 

If ${\mathfrak M}_j$ is a Banach space, linearity of the maps is a subject of  consideration. Let $\isonly_{\mc}(\mathfrak{M}_1,\mathfrak{M}_2)$ denote the set of all surjective {\it complex-linear} isometries. There exists an extensive literature on 2-local isometries in $\isonly_{\mc}(\mathfrak{M}_1,\mathfrak{M}_2)$ and 2-iso-reflexivity of $\isonly_{\mc}(\mathfrak{M}_1,\mathfrak{M}_2)$ (see, for example, \cite{ahhf,bjm,gy,hmto,jlp,jvvv,lpww,molbook,mol2l}). Note that Hosseini showed that a 2-local {\it real-linear} isometry is in fact a surjective real-linear isometry on the algebra of $n$-times continuously differentiable functions on the interval $[0,1]$ with a certain norm \cite[Theorem 3.1]{hoseini}. She described  that a 2-local real-linear isometry defined on the Banach algebra $C(X)$ of all complex-valued continuous functions on a compact Hausdorff space $X$ which is separable and first countable is in fact a surjective real-linear isometry on $C(X)$ \cite[Proposition 3.2]{hoseini}. At this point we emphasize that the situation is very different from that for the problem of 2-local isometry. 
We do not know 
 if a 2-local isometry defined on $C[0,1]$ is a surjective isometry on $C[0,1]$ or not.
The problem of whether the group of all surjective isometries (without assuming linearity) of $C(X)$ is 2-local reflexive or not has been raised by Moln\'ar 
who has proved a related positive result concerning the group of all surjective isometries in the setting of operator algebras   \cite{mo2loc}.

In this paper we study 2-local reflexivity for $\operatorname{Iso}(\mathfrak{M}_1,\mathfrak{M}_2)$ and we consider the question whether every 2-local isometry necessarily belongs to $\operatorname{Iso}(\mathfrak{M}_1,\mathfrak{M}_2)$, where $\mathfrak{M}_j$ is a certain space of continuous functions. 


\section{Preliminaries}\label{sec2}
Let $X_j$ be a compact Hausdorff space for $j=1,2$. The algebra of all complex-valued continuous functions on $X_j$ is denoted by $C(X_j)$. The supremum norm is denoted by $\|\cdot\|_{\infty}$. In the remaining of the paper $E_j$ is a subspace of $C(X_j)$ which contains the constant functions and separates the points of $X_j$. For $c\in {\mathbb C}$ we write the constant function which takes the value $c$ by $c$. We assume that the norm $\|\cdot\|_j$ is defined on $E_j$ (not necessary complete) and it satisfies that $\|c\|_j=|c|$ for every $c\in {\mathbb C}$. We assume that $E_j$ is conjugate closed in the sense that $f\in E_j$ implies $\bar f\in E_j$, and that $\|f\|_j=\|\bar f\|_j$ for every $f\in E_j$. For an $\epsilon\in \{\pm 1\}$ and $f\in E_j$, $[f]^{\epsilon}=f$ if $\epsilon =1$ and $[f]^\epsilon=\bar f$ if $\epsilon=-1$. Let $M(E_1,E_2)$ be the set of all maps from $E_1$ into $E_2$. 
Note that we say a map is a surjective isometry if it is just a distance preserving map, we do not assume complex nor real linearity on it. We abbreviate $\operatorname{Iso}(E_j,E_j)$ by $\operatorname{Iso}(E_j)$. 
Let $\Pi$ denotes a non-empty set of (not always all) homeomorphisms from $E_2$ onto $E_1$. Let 
\begin{multline*}
\text{$G_{\Pi}(E_1,E_2)=\{T\in M(E_1,E_2):$ there exists a $\lambda\in E_2$,}\\
\text{an $\alpha\in {\mathbb C}$ of unit modulus, a $\pi\in \Pi$, and an $\epsilon\in \{\pm1\}$} \\
\text{such that $T(f)=\lambda + \alpha [f\circ \pi]^\epsilon$ for every $f\in E_1$\}}.
\end{multline*}
We abbreviate $G_{\Pi}(E_j,E_j)$ by $G_{\Pi}(E_j)$. 
We usually abbreviate $G_{\Pi}(E_1,E_2)$ by $G_{\Pi}$ if $E_1$ and $E_2$ are clear from the context. 
Let $\operatorname{Id}_{[0,1]}=\pi_0:[0,1]\to [0,1]$ be the identity function and $\pi_1=1-\operatorname{Id}_{[0,1]}$. Put $\Pi_0=\{\pi_0,\pi_1\}$. 
Kawamura, Koshimizu and Miura \cite{kkm} (cf. \cite{mt}) proved that $G_{\Pi_0}(C^1[0,1])=\operatorname{Iso}(C^1[0,1],\|\cdot\|)$ with respect to several norms including $\|\cdot\|_{\Sigma}$. 
For a compact metric space $K$, let 
\[
\lip(K)=\left\{f\in C(K):L_f=\sup_{x\ne y}\frac{|f(x)-f(y)|}{d(x,y)}<\infty\right\}
\]
with the norm $\|f\|_{\Sigma}=\|f\|_{\infty}+L(f)$ for $f\in \lip(K)$. We say that $L_f$ is the Lipschitz constant for $f$. With this norm $\lip(K)$ is a unital semisimple commutative Banach algebra. 
We show in section \ref{sec4} that $G_{\Pi}(\operatorname{Lip}(K_1),\lip(K_2))=\operatorname{Iso}(\operatorname{Lip}(K_1),\lip(K_2))$, where $\Pi$ is the set of all surjective isometries from $K_2$ onto $K_1$.  

Let 
\begin{multline*}
\text{$W_j=\{f\in E_j:$ if $S:{\mathbb C}\to {\mathbb C}$ is an isometry}\\
\text{ and $S(f(X_j))=f(X_j)$, then $S$ is the identity function\}}.
\end{multline*}
Suppose that $S:{\mathbb C}\to {\mathbb C}$ is an isometry. It is well known that there exists $a,b\in {\mathbb C}$ with $|a|=1$ such that $S(z)=b+az$, ($z\in {\mathbb C}$) or $S(z)=b+a\bar z$, ($z\in {\mathbb C}$). The first case of $S$ is a parallel translation by $b$ if $a=1$ and $S$ is a rotation around $b/(1-a)$. Hence there is no fixed point if $a=1$ and $b\ne 0$, and $b/(1-a)$ is the unique fixed point if $a\ne 1$ for the first case. For the second case, denoting one of the square root of $a$ by $a^\frac12$, $S$ is a symmetric translation with respect to the line $t\mapsto a^\frac12t+i(\operatorname{Im}(\overline{a^\frac12}b)/2)$ ($t\in {\mathbb R}$) followed by the parallel translation by $\operatorname{Re}(\overline{a^\frac12}b)/2)$ to the direction of $a^\frac12$. Hence the fixed points exist and they are all the points on line $t\mapsto a^\frac12t+i(\operatorname{Im}(\overline{a^\frac12}b)/2)$ if and only if $\operatorname{Re}(\overline{a^\frac12}b)/2)=0$. We will prove that $W_j$ for $E_j=C^1[0,1]$ is uniformly dense in $C[0,1]$ (see Proposition \ref{PsubsetBarU}).
\begin{lemma}\label{lemma0}
If $T\in \M$ is 2-local in $\gp\cap\iso$, then $T$ is an isometry with respect to the metric induced by the norm $\|\cdot\|_j$; $\|T(f)-T(g)\|_2=\|f-g\|_1$ for every pair $f,g\in E_1$. The map $T$ is also an isometry with respect to the supremum norm $\|\cdot\|_\infty$.
\end{lemma}
\begin{proof}
Let $f,g\in E_1$. Then there exists $T_{f,g}\in \gp\cap\iso$ such that 
\begin{equation}\label{77}
T(f)=T_{f,g}(f),\qquad T(g)=T_{f,g}(g).
\end{equation}
As $T_{f,g}$ is an isometry we have
\[
\|T(f)-T(g)\|_2=\|T_{f,g}(f)-T_{f,g}(g)\|_2=\|f-g\|_1.
\]
Thus $T$ is an isometry. 
As $T_{f,g}\in \gp\cap\iso$ there exists a $\lambda_{f,g}\in E_2$, an $\alpha_{f,g}\in \mc$ of unit modulus, a $\pi\in \Pi$ and an $\epsilon_{f,g}\in \{\pm 1\}$ such that 
\begin{equation}\label{78}
T_{f,g}(h)=\lambda_{f,g}+\alpha_{f,g}[h\circ \pi]^{\epsilon_{f,g}}, \quad h\in E_1.
\end{equation}
Then by \eqref{77} and \eqref{78} we observe that
\begin{multline*}
\|T(f)-T(g)\|_{\infty}=\|T_{f,g}(f)-T_{f,g}(g)\|_{\infty}\\
= \|\alpha_{f,g}[f\circ \pi]^{\epsilon}-\alpha_{f,g}[g\circ \pi]^{\epsilon}\|_{\infty}\\
=\|[f\circ \pi]^{\epsilon}-[g\circ \pi]^{\epsilon}\|_{\infty}=\|f-g\|_{\infty}
\end{multline*}
since $\pi$ is a surjection. Thus $T$ is an isometry with respect to the supremum norm. 
\end{proof}
\begin{prop}\label{pro1}
Suppose that $T\in \M$ is 2-local in $\gpiee$. Then there exists $\epsilon\in \{\pm 1\}$ and $\alpha\in {\mathbb C}$ of unit modulus such that for every $f\in W_1$ there exists a homeomorphism $\pi_f\in \Pi$ such that
\[
T(f)=T(0)+\alpha[f\circ\pi_f]^\epsilon.
\]
\end{prop}
Note that if we proved that $\pi_f$ did not depend on $f$, then the map $T$ were surjective, hence $T\in \iso$ by the Mazur-Ulam theorem. But it is not the case in general (cf. \cite[Theorem 2.3]{hmto}); $\gp\cap\iso$ needs not be 2-local reflexive in $\M$.
\begin{proof}[Proof of Proposition \ref{pro1}]
Put $T_0=T-T(0)$. We infer by a simple calculation that $T_0$ is 2-local in $\gpiee$. Let $h\in E_1$. Since $T_0$ is 2-local in $\gpiee$, there exist $T_{h,0}\in \gpiee,\lambda_{h,0}\in E_2, \alpha_{h,0}\in \mc$ with $|\alpha_{h,0}|=1$, a homeomorphism $\pi_{h,0}:X_2\to X_2$ and and $\epsilon_{h,0}\in \{\pm 1\}$ such that 
\begin{eqnarray}\label{-1}
T_0(h)&=&T_{h,0}(h)=\lambda_{h,0}+\alpha_{h,0}[h\circ\pi_{h,0}]^{\epsilon_{h,0}}, \\
0=T_0(0)&=&T_{h,0}(0)=\lambda_{h,0}.\nonumber
\end{eqnarray}
Note that $T_{h,0}$ is represented by $T_{h,0}(\cdot)=\alpha_{h,0}[\cdot\circ\pi_{h,0}]^{\epsilon_{h,0}}$ on $E_1$.
Hence $\lambda_{h,0}=0$ and 
\begin{equation}\label{(0)}
T_0(h)=\alpha_{h,0}[h\circ \pi_{h,0}]^{\epsilon_{h,0}}.
\end{equation}
In particular, if $h=c\in \mc$, then 
\[
T_0(c)=\alpha_{c,0}[c]^{\epsilon_{h,0}}.
\]
Thus we obtain $T_0(\mc)\subset \mc$. For every pair $c,d\in \mc$, here exists $T_{c,d}\in \G$ such that 
\[
T_0(c)=T_{c,d}(c),\quad T_0(d)=T_{c,d}(d).
\]
As $T_{c,d}\in \iso$ we infer that
\begin{multline*}
|T_0(c)-T_0(d)|=\|T_0(c)-T_0(d)\|_2\\
=\|T_{c,d}(c)-T_{c,d}(d)\|_2=\|c-d\|_1=|c-d|.
\end{multline*}
As $c,d$ are arbitrary, we have that $T_0|_{\mc}:\mc\to\mc$ is an isometry.
Applying a well known result about the form of an isometry on $\mc$, there exists an $\alpha\in \mc$ such that
\begin{equation}\label{(1)}
T_0(z)=\alpha z,\quad z\in \mc,
\end{equation}
or
\begin{equation}\label{(2)}
T_0(z)=\alpha \bar z,\quad z\in \mc. 
\end{equation}
Put $T_1=[\bar{\alpha}T_0]^{\epsilon}$, where $\epsilon=1$ if \eqref{(1)} holds and $\epsilon=-1$ if \eqref{(2)} holds. Since $E_2$ is conjugate closed, $T_1$ is well defined in the second case. Since $\|f\|_2=\|\bar f\|_2$ for every $f\in E_2$, it is a routine work to see that  $T_1$ is 2-local in $\gpiee$. By the definition of $T_1$ we infer that $T_1(z)=z$ for every $z\in \mc$. We will prove that for every $f\in W_1$ there exists a homeomorphism $\pi_f:X_2\to X_1$ such that 
\[
T_1(f)=f\circ\pi_f.
\]
If it is proved, then we have
\[
T(f)=T(0)+\alpha[f\circ\pi_f]^{\epsilon},
\]
the desired form.

Let $f\in W_1$ and $c\in \mc$. As $\tone$ is 2-local in $\G$, there exists $\lambda_{f,c}\in E_2$ and $\alpha_{f,c}\in \mc$ of modulus 1, a homeomorphism $\pi_{f,c}:X_2\to X_1$, and $\epsilon_{f,c}\in \{\pm 1\}$ such that
\begin{equation}\label{(3)}
T_0(f)=T_{f,c}(f)=\lambda_{f,c}+\alpha_{f,c}[f\circ \pi_{f,c}]^{\epsilon_{f,c}},\end{equation}
\begin{equation}\label{(4)}
c=T_0(c)=\lambda_{f,c}+\alpha_{f,c}[c]^{\epsilon_{f,c}},
\end{equation}
where $T_{f,c}(\cdot)=\lambda_{f,c}+\alpha_{f,c}[\cdot\circ\pi_{f,c}]^{\epsilon_{f,c}}$. 
By \eqref{(4)} we infer that $\lambda_{f,c}$ is a constant. By comparing \eqref{(0)} for $f$ with \eqref{(3)} we get
\[
\alpha_{f,0}[f\circ\pi_{f,0}]^{\epsilon_{f,0}}=\lambda_{f,c}+\alpha_{f,c}[f\circ\pi_{f,c}]^{\epsilon_{f,c}},
\]
hence
\begin{equation}\label{(5)}
f\circ \pi_{f,0}=\left[\overline{\alpha_{f,0}}\lambda_{f,c}+\overline{\alpha_{f,0}}\alpha_{f,c}[f\circ\pi_{f,c}]^{{\epsilon}_{f,c}}\right]^{\epsilon_{f,0}}.
\end{equation}
Considering the range of the both side of \ref{(5)} we get
\begin{equation}\label{*}
f(X_1)=\left[\overline{\alpha_{f,0}}\lambda_{f,c}+\overline{\alpha_{f,0}}\alpha_{f,c}[f(X_1)]^{{\epsilon}_{f,c}}\right]^{\epsilon_{f,0}}.
\end{equation}
We have four possibility depending on (i)$\epsilon_{f,0}=1$ and $\epsilon_{f,c}=1$; (ii)$\epsilon_{f,0}=1$ and $\epsilon_{f,c}=-1$; (iii)$\epsilon_{f,0}=-1$ and $\epsilon_{f,c}=-1$; (iv)$\epsilon_{f,0}=-1$ and $\epsilon_{f,c}=1$. Then we have that at least one of the following (i) through (iv) holds.
\[
{\rm(i)}\quad f(X_1)=\overline{\alpha_{f,0}}\lambda_{f,c}+\overline{\alpha_{f,0}}\alpha_{f,c}f(X_1);
\]
\[
{\rm(ii)}\quad f(X_1)=\overline{\alpha_{f,0}}\lambda_{f,c}+\overline{\alpha_{f,0}}\alpha_{f,c}\overline{f(X_1)};
\]
\[
{\rm(iii)}\quad f(X_1)=\alpha_{f,0}\overline{\lambda_{f,c}}+\alpha_{f,0}\overline{\alpha_{f,c}}f(X_1);
\]
\[
{\rm(iv)}\quad f(X_1)=\alpha_{f,0}\overline{\lambda_{f,c}}+\alpha_{f,0}\overline{\alpha_{f,c}}\overline{f(X_1)}.
\]
corresponding to the cases (i) through (iv) respectively. Let $S_j:\mc\to\mc$ be defined by $S_1(z)=\overline{\alpha_{f,0}}\lambda_{f,c}+\overline{\alpha_{f,0}}\alpha_{f,c}z$, $z\in \mc$; $S_2(z)=\overline{\alpha_{f,0}}\lambda_{f,c}+\overline{\alpha_{f,0}}\alpha_{f,c}\overline{z}$, $z\in \mc$; $S_3(z)=\alpha_{f,0}\overline{\lambda_{f,c}}+\alpha_{f,0}\overline{\alpha_{f,c}}z$, $z\in \mc$; $S_4(z)=\alpha_{f,0}\overline{\lambda_{f,c}}+\alpha_{f,0}\overline{\alpha_{f,c}}\overline{z}$, $z\in \mc$. Then $S_j$ is an isometry on $\mc$ for $j=1,2,3,4$. Using $S_j$ we rewrite (i) by $f(X_1)=S_1(f(X_1))$, (ii) by $f(X_1)=S_2(f(X_1))$, (iii) by $f(X_1)=S_3(f(X_1))$ and (iv) by $f(X_1)=S_4(f(X_1))$. As $f\in W_1$ we see that (ii) and (iv) do not occur. We have $\overline{\alpha_{f,0}}\lambda_{f,c}=0$ and $\overline{\alpha_{f,0}}\alpha_{f,c}=1$ for the case (i). Hence $\lambda_{f,c}=0$ and $\alpha_{f,0}=\alpha_{f,c}$ in this case. In the same way we have $\lambda_{f,0}=0$ and $\alpha_{f,0}=\alpha_{f,c}$ for the case (iii).
We conclude that only (i) or (iii) occur, and  in any case 
\[
\text{$\lambda_{f,c}=0$ and $\alpha_{f,0}=\alpha_{f,c}$}.
\]

To prove $\alpha_{f,0}=1$, put $c=1$. Then (i) or (iii) occur. If (i) occurs, then 
\[
1=T_0(1)=T_{f,1}(1)=\alpha_{f,1}1=\alpha_{f,0}1.
\]
If (iii) occurs, then
\[
1=T_0(1)=T_{f,1}(1)=\alpha_{f,1}\bar 1=\alpha_{f,0}1.
\]
It follows that 
\begin{equation}\label{alpha=0}
\alpha_{f,0}=1.
\end{equation}

Next we prove that (iii) does not occur for any $c\in \mc$ and $\epsilon_{f,0}=1$. Suppose that (iii) occurs for some $c_0\in \mc$. Then we have 
\[
c_0=T_0(c)=T_{f,c_0}(c_0)=\alpha_{f,c_0}\overline{c_0}=\alpha_{f,0}\overline{c_0}=\overline{c_0}.
\]
Hence $c_0$ is a real number. We also have
\begin{equation}\label{(6)}
T_0(f)=T_{f,c_0}(f)=\overline{f\circ\pi_{f,c_0}}
\end{equation}
since $\alpha_{f,c_0}=\alpha_{f,0}=1$. On the other hand, (iii) does not occur for $c_0+i$ since $c_0+i$ is not a real number. Thus (i) occurs and
\begin{equation}\label{(7)}
T_0(f)=T_{f,c_0+i}(f)=f\circ\pi_{f,c_0+i}.
\end{equation}
It follows by \eqref{(6)} and \eqref{(7)} that 
\[
\overline{f(X_1)}=f(X_1),
\]
which is a contradiction since $f\in W_1$. We conclude that (iii) does not occur for any $c\in \mc$. It follows that only (i) occurs, hence we have 
\begin{equation}\label{epsilon=1}
\epsilon_{f,0}=1.
\end{equation}
Then by \eqref{(0)} for $h=f$ we have
\[
T_0(f)=T_{f,0}(f)=f\circ \pi_{f,0}.
\]
Letting $\pi_f=\pi_{f,0}$ we have the conclusion.
\end{proof}
\section{Spaces of continuous functions on $[0,1]$}\label{sec3}
In this section $A_j$ ($j=1,2$) is a subspace of $C[0,1]$ and a superspace of $C^1[0,1]$, the space of all complex-valued continuously differentiable functions on the interval $[0,1]$. We assume that $A_j$ satisfies the following three conditions.
\begin{itemize}
\item[1)] $A_j$ is conjugate-closed in the sense that $f\in A_j$ implies $\bar f\in A_j$;
\item[2)] The norm  $\|\cdot\|_j$ on $A_j$ satisfies  that $|c|=\|c\|_j$ for every $c\in \mc$; 
\item[3)] The norm  $\|\cdot\|_j$ on $A_j$ satisfies  that $\|f\|_j=\|\bar f\|_j$ for every $f\in A_j$, where $\bar{\cdot}$ denotes the complex-conjugation;
\end{itemize}
We do not assume the completeness of $\|\cdot\|_j$. 
The space $A_j$ satisfies the conditions for $E_j$ in the previous section. The difference between $A_j$ and $E_j$ is that $A_j$ is defined on $[0,1]$ and we assume that $C^1[0,1]\subset A_j$. The spaces $(C^1[0,1],\|\cdot\|_{\Sigma})$, $(C^1[0,1],\|\cdot\|_M)$, $(\operatorname{Lip}[0,1],\|\cdot\|_{\Sigma})$, $(\operatorname{Lip}[0,1],\|\cdot\|_{M})$ and $(C[0,1],\|\cdot\|_{\infty})$ are typical examples of $A_j$. Recall that $\pi_0$ is the identity function on $[0,1]$ and $\pi_1=1-\pi_0$, $\Pi_0=\{\pi_0,\pi_1\}$. Kawamura, Koshimizu and Miura \cite{kkm} studies the space $C^1[0,1]$ with a variety of norms including $\|\cdot\|_{\Sigma}$ and $\|\cdot\|_M$.  Recall that
\begin{multline*}
\text{$W_1=\{f\in A_j:$ if $S:{\mathbb C}\to {\mathbb C}$ is an isometry}\\
\text{ and $S(f(X_j))=f(X_j)$, then $S$ is the identity function\}}.
\end{multline*}
Put $P=\{p+iq:$$p$ and $q$ are polynomials of real-coefficients.$\}$. Many polynomials are in $W_1$, but some are not. For example $(t-\frac12)^4+i(t-\frac12)^3\not\in W_1$. We do no know if $P\cap W_1$ is uniformly closed in $P$ or not. We have that following. Let $cl(\cdot)$ denote the uniform closure on $[0,1]$.
\begin{prop}\label{PsubsetBarU}
We have $P\subset cl(W_1)$. Hence $cl(W_1)=C[0,1]$.
\end{prop}
\begin{proof}
Let $f=p+iq\in P$ and $\varepsilon>0$. If $p$ is not a constant, then put $p_{\varepsilon}=p$. If $p$ is a constant, then put $p_{\varepsilon}=p+\varepsilon \pi_0$, where $\pi_0$ is the identity function on $[0,1]$. Let $l$ be any positive integer greater than both of the degree of $p$ and $q$. Put $q_{\varepsilon}=q+\varepsilon \pi_0^{l}$. Then $p_{\varepsilon}$ is not a constant and there is no pair of complex numbers $c$ and $d$ such that $p_{\varepsilon}=c\qe +d$ since the degree of the each side of the equation is different. We prove that $p_{\varepsilon}+iq_{\varepsilon}\in cl(W_1)$. Then $p+iq\in cl(W_1)$ follows since $p_{\varepsilon}+iq_{\varepsilon}$ uniformly converges on $[0,1]$ to $p+iq$ as $\varepsilon\to 0$. 

Since $p_{\varepsilon}$ is a non-constant polynomial, there exists a positive integer $m_0$ such that $p_{\varepsilon}'(\frac{1}{m})\ne 0$ for every $m\ge m_0$. Let $m\ge m_0$. Put
\begin{equation*}
f_m(t)=
\begin{cases}
iw\left(\frac{1}{m}-t\right)
+\left(p_{\varepsilon}'\left(\frac{1}{m}\right)+iq_{\varepsilon}'\left(\frac{1}{m}\right)\right)\left(t-\frac{1}{m}\right)\\
\hspace{2cm}+\pep\left(\frac{1}{m}\right)
+i\qe\left(\frac{1}{m}\right), & 0\le t\le\frac{1}{m}, \\
(\pep+i\qe)(t), & \frac{1}{m}\le t\le 1, 
\end{cases}
\end{equation*}
where 
\begin{equation*}
w(t)=
\begin{cases}
0,& t=0 \\
t^3\sin \frac{1}{t},& 0<t\le 1
\end{cases}
\end{equation*}
Then $f_m\in C^1[0,1]$ for every $m\ge m_0$. It is a routine work to prove that $f_m$ converges uniformly to $p+iq$ on $[0,1]$ and a proof is omitted. We prove that $f_m\in W_1$ for every $m\ge m_0$.

Let $K$ be a real number. We look at the number of the points $t$ on $[0,1]$, where $f_m(t)$ is a tangent point of a tangent line of $f_m([0,1])$ whose slope is $K$. The curve $f_m([0,1])$ has a tangent line of the slope $K$ at the tangent point $f_m(t)$ if and only if 
\[
K=\lim_{\delta\to 0}\frac{\Imarg f_m(t+\delta)-\Imarg f_m(t)}{\Real f_m(t+\delta)-\Real f_m(t)}.
\]
Suppose that $0\le t\le \frac{1}{m}$. Since $\Real f_m(t)=\pedonem\left(t-\frac{1}{m}\right)+\pep\left(\frac{1}{m}\right)$ and $\Imarg f_m(t)=w\left(\frac{1}{m}-t\right)+\qe'\left(\frac{1}{m}\right)(t-\frac{1}{m})+\qe \left(\frac{1}{m}\right)$, we have 
\[
\lim_{\delta\to 0}\frac{\Imarg f_m(t+\delta)-f_m(t)}{\Real f_m(t+\delta)-\Real f_m(t)}=\frac{-\wdonem+\qedonem}{\pedonem},\quad0\le t\le \frac{1}{m}.
\]
Hence the curve $f_m([0,1])$ has a tangent line of the slope $K$ at the tangent point $f_m(t)$ for $0\le t\le \frac{1}{m}$ if and only if
\begin{equation}\label{*}
\frac{-\wdonem+\qedonem}{\pedonem}=K.
\end{equation}
If $K\ne \frac{\qedonem}{\pedonem}$, the number of such points $0\le t\le \frac{1}{m}$ is at most finite.
(The reason is as follows. Suppose that $\frac{\qedonem}{\pedonem}\ne K=\frac{-\wdonem+\qedonem}{\pedonem}$. Then
\begin{equation}\label{(10)}
\wdonem
=\pedonem\left(\frac{\qedonem}{\pedonem}-K\right).
\end{equation}
On the other hand, a simple calculation shows that
\begin{equation}\label{big}
\left|\wdonem\right|\le 4\left(\frac{1}{m}-t\right).
\end{equation}
We have $\pedonem\left(\frac{\qedonem}{\pedonem}-K\right)\ne 0$ since $\frac{\qedonem}{\pedonem}\ne K$. 
By \eqref{big} there is no $t\le \frac{1}{m}$ with 
\[
\frac{1}{m}-t< \frac14\left|\pedonem\left(\frac{\qedonem}{\pedonem}-K\right)\right|
\]
such that \eqref{(10)} holds. It is easy to see that the number of $t\ge 0$ with 
\[
\frac14\left|\pedonem\left(\frac{\qedonem}{\pedonem}-K\right)\right|\le 
\frac{1}{m}-t
\] 
such that \eqref{(10)} holds is at most finite.) On the other hand if $K=\frac{\qedonem}{\pedonem}$, then by \eqref{*} we infer that $\wdonem=0$. By a calculation, for every positive integer $k$ there exists a unique $k\pi<s_k<k\pi+\pi/2$ such that $w'\left(\frac{1}{s_k}\right)=0$. Letting $t_k=\frac{1}{m}-\frac{1}{s_k}$ we have $w'\left(\frac{1}{m}-t_k\right)=0$. Thus $K=\frac{\qedonem}{\pedonem}$ for $0\le t<\frac{1}{m}$ if and only if $t=t_k$ for some positive integer $k$. As $w'(0)=0$, we see that $\wdonem=\frac{\qedonem}{\pedonem}$ if $t=\frac{1}{m}$. 
We conclude that the set of the all points in $f_m([0,\frac{1}{m}])$ at which $f_m([0,1])$ has a tangent line with the slope $\frac{\qedonem}{\pedonem}$ is $\{f_m(t_n)\}_{n\ge N}\cup \left\{f_m\left(\frac{1}{m}\right)\right\}$, where $N=\min\left\{k:\frac{1}{m}>t_k\right\}$.

Suppose that $\frac{1}{m}<t\le 1$. We have $\Real f_m(t)=\pep(t)$ and $\Imarg f_m(t)=\qe(t)$. Therefore we have
\[
\frac{\Imarg f_m(t+\delta)-\Imarg f_m(t)}{\Real f_m(t+\delta)-\Real f_m(t)}=
\frac{\qe (t+\delta)-\qe(t)}{\pep(t+\delta)-\pep(t)}.
\]
Hence the curve $f_m([0,1])$ has a tangent line of the slope $K$ at $f_m(t)$ for $\frac{1}{m}< t\le 1$ if and only if 
\[
\lim_{\delta\to 0}\frac{\qe(t+\delta)-\qe(t)}{\pep(t+\delta)-\pep(t)}=K.
\]
If $\pep(t)\ne 0$, then $\frac{\qe'(t)}{\pep'(t)}=K$. The number of such points $t\in \left.\left(\frac{1}{m},1\right.\right]$ is at most finite. Suppose not. Then $\qe'(t)=K\pep'(t)$ for infinitely many $t$, hence $\qe'=K\pep'$ on the interval $\left(\left.\frac{1}{m},1\right.\right].$ since $\pep'$ and $\qe'$ are polynomials. It follows that $\qe=K\pep +c$ for some $c\in \mc$, which contradicts to our hypothesis on $\pep$ and $\qe$. We obtain that the number of $t\in \left.\left(\frac{1}{m},1\right.\right]$ such that $\frac{\qe'(t)}{\pep'(t)}=K$ is at most finite. The number of $t\in \left.\left(\frac{1}{m},1\right.\right]$ such that $\pep'(t)=0$ is at most finite since $\pep$ is a polynomial. We conclude that the number of point $t$ such that $f_m([0,1])$ has a tangent line of the slope $K$ at $f_m(t)$ 
is at most finite.
In a way similar we see that the number of $t\in \left.\left(\frac{1}{m},1\right.\right]$ such that $f_m([0,1])$ has a tangent line at $f_m(t)$ which is parallel to the imaginary axis is at most finite.
If $\pep'(t)\ne 0$, then $f_m([0,1])$ has a tangent line which is not parallel to the imaginal axis. Hence if $f_m([0,1])$ has a tangent line with a tangent point at $f_m(t)$ which is parallel to the imaginary axis, then $\pep'(t)=0$. Thus the number of such points is at most finite.

We conclude that for a real number $K\ne\frac{\qedonem}{\pedonem}$ the number of $t\in [0,1]$ such that $f_m([0,1])$ has a tangent line of the slope $K$ at $f_m(t)$ is at most finite; the number of $t\in [0,1]$ such that $f_m([0,1])$ has a tangent line which is parallel to the imaginal axis at $f_m(t)$ is at most finite; the number of $t\in [0,1]$ such that $f_m([0,1])$ has a tangent line of the slope $\frac{\qedonem}{\pedonem}$ at $f_m(t)$ is countable, the number of such $t\in \left.\left(\frac{1}{m},1\right.\right]$ is at most finite, say $\{t_{-n}\}_{n=1}^l$, there is a sequence $\{t_k\}_{k\ge N}$ in $[0,\frac{1}{m})$ of such points that converges to $\frac{1}{m}$. Denote them as $\{t_k\}=\{t_{k}\}_{k\ge N}\cup\{t_{\infty}=\frac{1}{m}\}\cup\{t_{-n}\}_{n=1}^l$.

Suppose that $S:\mc\to\mc$ is an isometry such that $S(f_m([0,1]))=f_m([0,1])$. We prove that $S$ is the identity so that $f_m\in W$. Since there are $a,b\in \mc$ with $|a|=1$ such that $S(z)=b+az$, $z\in \mc$ or $S(z)=b+a\bar z$, $z\in \mc$, $S(\ell_1)$ and $S(\ell_2)$ are parallel for every pair of parallel lines $\ell_1$ and $\ell_2$ in $\mc$. Thus the parallel tangent line at $\{t_k\}$  are translated by $S$ as a parallel tangent line. Hence we get
\[
\{S(f_m(t_k))\}=\{f_m(t_k)\}.
\]
As $S$ is an isometry the unique cluster points $t_{\infty}$ of $\{t_k\}$ translates to $t_{\infty}$ by $S$; $S(t_{\infty})=t_{\infty}$. As $\{f_m(t_{-k})\}_{k=1}^l$ is discrete, there is a positive integer $M$ such that $\{S(f_m(t_k))\}_{k\ge M}\subset \{f_m(t_k)\}_{k\ge N}$. Hence if $n\ge M$, then there is an $n_1\ge N$ such that 
\[
|f_m(t_{\infty})-f_m(t_n)|=|S(f_m(t_{\infty}))-S(f_m(t_n))|
=|f_m(t_{\infty})-f_m(t_{n_1})|.
\]
If $n,n_1\ge N$ and $n\ne n_1$, then by the definition of $t_n$ and $t_{n_1}$ we have 
\[
|f_m(t_{\infty})-f_m(t_n)|\ne |f_m(t_{\infty})-f_m(t_{n_1})|.
\]
It follows that $t_n=t_{n_1}$. Since $S:\mc\to\mc$ is an isometry, the set of the fixed point of $S$ is empty or a singleton or points on a straight line if $S$ is not the identity. As $\{f_m(t_n)\}_{n\ge M}$ is a set of fixed points which are not on the line since $w\left(\frac{1}{m}-t_n\right)>0$ when $n$ is an even number and $w\left(\frac{1}{m}-t_n\right)<0$ when $n$ is an odd number. Thus we conclude that $S$ is an isometry.

By the Weierstrass approximation theorem we see that $cl(W_1)=C[0,1]$.
\end{proof}
\begin{theorem}\label{2-localgeneral}
$\gpa$ is 2-local reflexive in $M(A_1,A_2)$.
\end{theorem}
\begin{proof}
Suppose that $T\in M(A_1,A_2)$ is 2-local in $\gpa$. Then by Proposition \ref{pro1} there exist an $\alpha\in \mc$ of unit modulus and $\epsilon\in \{\pm 1\}$ which satisfies that for every $f$ in $W_1$, 
there exists a $\pi_f\in \Pi_0$ such that 
\[
T(f)=T(0)+\alpha[f\circ \pi_f]^{\epsilon}.
\]
We prove that $\pi_f$ is independent of $f\in W_1$. Let $T_1\in M(A_1,A_2)$ be defined by 
\[
T_1(h)=[\bar \alpha(T(h)-T(0))]^{\epsilon}, \quad h\in A_1.
\]
Then $T_1$ is 2-local in $\gpa$. As in the same way in the proof of Proposition \ref{pro1} there exists a $T_{f,0}\in \gpa$ such that
\begin{equation}\label{(19)}
T_1(f)=T_{f,0}(f)=f\circ\pi_f
\end{equation}
for every $f\in W_1$. Let $\varepsilon>0$ is given. Then $g_{\varepsilon}=\pi_0+i\varepsilon\pi_0^2 \in W_1$. Hence there exists $T_{g_{\varepsilon},0}\in \gpa$ and $\pi_{\varepsilon}\in \Pi_0$ such that
\begin{equation}\label{(11)}
T_1(g_{\varepsilon})=T_{g_{\varepsilon,0}}(g_{\varepsilon})=g_{\varepsilon}\circ \pi_{\varepsilon}.
\end{equation}
Note that $T_{g_{\varepsilon,0}}(h)=h\circ \pi_{\varepsilon}$ for every $h\in A_1$ by the proof of Proposition \ref{pro1}. (In fact, due to the note just after \eqref{-1} we have $T_{g_{\varepsilon},0}(h)=\alpha_{g_{\varepsilon,0}}[h\circ \pi_{g_{\varepsilon,0}}]^{\epsilon_{g_{\varepsilon},0}}$ for $h\in A_1$. As $g_{\varepsilon}\in W_1$, we have by \eqref{alpha=0} and \eqref{epsilon=1} and letting $\pi_{g_{\varepsilon,0}}=\pi_{\varepsilon}$ we have $T_{g_{\varepsilon,0}}(h)=h\circ \pi_{\varepsilon}$ for every $h\in A_1$.)
We prove that there exists an $\varepsilon>0$ such that $\pi_{\varepsilon}=\pi_{\varepsilon'}$ for every $0<\varepsilon,\varepsilon'<\varepsilon_0$. Suppose not. Then there exist sequences $\{\varepsilon_n\}$ and $\{\varepsilon_n'\}$ of positive real numbers which converge to $0$ respectively such that $\pi_{\varepsilon_n}\ne \pi_{\varepsilon_n'}$ for every $n$. By Lemma \ref{lemma0} $T_1$ is a isometry with respect to $\|\cdot\|_{\infty}$, hence we infer that
\[
\|T_1(g_{\varepsilon_n})-T_1(\varepsilon_n')\|_{\infty}=\|g_{\varepsilon_n}-g_{\varepsilon_n'}\|_{\infty}=\|\varepsilon_n-\varepsilon_n'\|=|\varepsilon_n-\varepsilon_n'|\to 0
\]
as $n\to \infty$. On the other hand, as $\pi_{\varepsilon_n}\ne\pi_{\varepsilon_n'}$ for every $n$ we have
\[
\|T_1(g_{\varepsilon_n})-T_1(\varepsilon_n')\|_{\infty}=
\|g_{\varepsilon_n}\circ \pi_{\varepsilon_n}-g_{\varepsilon_n'}\circ\pi_{\varepsilon_n'}\|_{\infty} \ge \|\pi_{\varepsilon_n}-\pi_{\varepsilon_n'}\|_{\infty}-\varepsilon_n-\varepsilon_n' \to 1
\]
as $n\to \infty$, which is a contradiction proving $\pi_{\varepsilon}=\pi_{\varepsilon'}$ for every $0<\varepsilon,\varepsilon'<\varepsilon_0$ for some positive $\varepsilon_0$. Put the common $\pi_{\varepsilon}$ as $\pi$. Letting $\varepsilon\to 0$ in \eqref{(11)} we get
\[
T_1(\pi_0)=\pi_0\circ\pi.
\]
We prove that
\begin{equation}\label{(100)}
T_1(f)=f\circ \pi
\end{equation}
for every $f\in W$. We show a proof for the case where $\pi=\pi_0$. A proof for the case where $\pi=\pi_1$ is similar, and is omitted.

If $T_1(f)=f$ for every $f\in W_1$ is proved, then it turns out that $T_1$ is a surjective isometry. The reason is as follows. For a sufficiently small positive $\varepsilon$, we have proved $\pi_{\varepsilon}=\pi_0$ since we assume $\pi=\pi_0$ in \eqref{(100)}. Then $T_{g_{\varepsilon,0}}=T_1$ on $W_1$. Proposition \ref{PsubsetBarU} asserts that $W_1$ is uniformly dense in $C[0,1]$, hence in $A_1$. As $T_1$ is continuous with respect to $\|\cdot\|_{\infty}$ by Lemma \ref{lemma0}, we conclude that $T_1=T_{g_\varepsilon,0}$ on $A_1$. Since $T_{g_{\varepsilon}}$ is a surjective isometry we conclude that $T_1$ is a surjective isometry.
We prove that $T_1(f)=f$ for every $f\in W_1$. To prove it, suppose that there exists a $f_0\in W_1$ such that $T_1(f_0)\ne f_0$. Then by \eqref{(19)} we have
\begin{equation}\label{(30)}
T_1(f_0)=f_0\circ\pi_1.
\end{equation}
As $T_1$ is 2-local in $\gpa$, there exists a $\lambda_{f_0,\pi_0}\in A_2$, $\alpha_{f_0,\pi_0}\in \mc$ of unit modulus and $\epsilon_{f_0,\pi_0}\in \{\pm 1\}$ such that one of 
\begin{eqnarray}\label{(13)}
f_0\circ \pi_1=T_1(f_0)=\lambda_{f_0,\pi_0}+ \alpha_{f_0,\pi_0}[f_0]^{\epsilon_{f_0,\pi_0}},\\
\pi_0=T_1(\pi_0)=\lambda_{f_0,\pi_0}+\alpha_{f_0,\pi_0}\pi_0 \nonumber
\end{eqnarray}
and
\begin{eqnarray}\label{(13')}
f_0\circ\pi_1=T_1(f_0)=\lambda_{f_0,\pi_0}+\alpha_{f_0,\pi_0}[f_0\circ\pi_1]^{\epsilon_{f_0,\pi_0}}, \\
 \pi_0=T_1(\pi_0)=\lambda_{f_0,\pi_0}+\alpha_{f_0,\pi_0}\pi_1\nonumber
\end{eqnarray}
holds. Thus
\begin{equation}\label{(51)}
\lambda_{f_0,\pi_0}(t)=(1-\alpha_{f_0,\pi_0})t,\quad t\in [0,1]
\end{equation}
when \eqref{(13)} occurs and 
\begin{equation}\label{(52)}
\lambda_{f_0,\pi_0}(t)=(1+\alpha_{f_0,\pi_0})t-\alpha_{f_0,\pi_0},\quad t\in [0,1]
\end{equation}
when \eqref{(13')} occurs. 

We will prove that both of \eqref{(13)} and \eqref{(13')} are impossible. Suppose that \eqref{(13)} occurs. Rewriting the first equation of \eqref{(13)} using \eqref{(51)} we get
\begin{equation}\label{(40)}
f_0(1-t)=(T_1(f_0))(t)=(1-\alpha_{f_0,\pi_0})t +\alpha_{f_0,\pi_0}[f_0(t)]^{\epsilon_{f_0,\pi_0}}, \quad t\in [0,1]
\end{equation}
Suppose that $\alpha_{f_0,\pi_0}=1$. Then 
\begin{equation}\label{(14)}
f_0(1-t)=f_0(t),\qquad t\in [0,1]
\end{equation}
or
\begin{equation}\label{(15)}
f_0(1-t)=\overline{f_0(t)},\qquad t\in [0,1].
\end{equation}
If \eqref{(14)} holds, then $(T_1(f_0))(t)=f_0(1-t)=f_0(t)$, $t\in [0,1]$ by \eqref{(30)}, which is against to our choice of $f_0$. Thus \eqref{(14)} does not hold. Suppose that \eqref{(15)} holds. Then $f_0([0,1])=\overline{f_0([0,1])}$  holds, which means that $f_0\not\in W_1$. Thus \eqref{(15)} does not hold. It follows that $\alpha_{f_0,\pi_0}\ne 1$. Suppose that $\varepsilon_{f_0,\pi_0}=1$ for \eqref{(40)}. Then we have
\begin{equation}\label{(41)}
f_0(1-t)=(1-\alpha_{f_0,\pi_0})t+\alpha_{f_0,\pi_0}f_0(t),\quad t\in [0,1].
\end{equation}
Changing $1-t$ by $t$ we have 
\begin{equation}\label{(42)}
f_0(t)=(1-\alpha_{f_0,\pi_0})(1-t)+\alpha_{f_0,\pi_0}f_0(1-t),\quad t\in [0,1].
\end{equation}
Applying \eqref{(41)} we have
\begin{equation}\label{(43)}
f_0(t)=(1-\alpha_{f_0,\pi_0})(1-t)+\alpha_{f_0,\pi_0}\left((1-\alpha_{f_0,\pi_0})t+\alpha_{f_0,\pi_0}f_0(t)\right),\quad t\in [0,1].
\end{equation}
As $\alpha_{f_0,\pi_0}\ne 1$ we infer that
\begin{equation}\label{(44)}
(1+\alpha_{f_0,\pi_0})f_0(t)=1-(1-\alpha_{f_0,\pi_0})t,\qquad t\in [0,1].
\end{equation}
If $\alpha_{f_0,\pi_0}=-1$, then we have that $0=1-2t$ for every $t\in [0,1]$, which is a contradiction. Hence $\alpha_{f_0,\pi_0}\ne -1$. Then by \eqref{(44)} we have 
\[
f_0(t)=\frac{1}{(1+\alpha_{f_0,\pi_0})}-\frac{(1-\alpha_{f_0,\pi_0})}{(1+\alpha_{f_0,\pi_0})}t,\qquad t\in[0,1].
\]
Hence $f_0\in W_1$, which is a contradiction.
Suppose that $\epsilon_{f_0,\pi_0}=-1$ for \eqref{(40)}. Then we have
\begin{equation}\label{(45)}
f_0(1-t)=(1-\alpha_{f_0,\pi_0})t+\alpha_{f_0,\pi_0}\overline{f_0(t)}, \quad t\in [0,1].
\end{equation}
Substituting $1-t$ by $t$ in \eqref{(45)}, we have
\begin{equation}\label{(46)}
f_0(t)=(1-\alpha_{f_0,\pi_0})(1-t)+\alpha_{f_0,\pi_0}\overline{f_0(1-t)},\quad t\in [0,1].
\end{equation}
Substituting \eqref{(45)} in \eqref{(46)} we get
\begin{equation}\label{(47)}
f_0(t)=(1-\alpha_{f_0,\pi_0})(1-t)+\alpha_{f_0,\pi_0}\overline{(1-\alpha_{f_0,\pi_0})t+\alpha_{f_0,\pi_0}\overline{f_0(t)}},\quad t\in [0,1].
\end{equation}
Hence we get
\[
0=\left(\alpha_{f_0,\pi_0}\overline{(1-\alpha_{f_0,\pi_0})}-(1-\alpha_{f_0,\pi_0})\right)t+ (1-\alpha_{f_0,\pi_0}), \qquad t\in [0,1].
\]
We get $\alpha_{f_0,\pi_0}=1$, which contradicts to $\alpha_{f_0,\pi_0}\ne 1$. We conclude that \eqref{(13)} does not occur.

Suppose that \eqref{(13')} holds. Rewriting \eqref{(13')} by applying \eqref{(52)} we get
\begin{equation}\label{(50)}
f_0(1-t)=(T_1(f_0))(t)=(1+\alpha_{f_0,\pi_0})t-\alpha_{f_0,\pi_0}+\alpha_{f_0,\pi_0}[f_0(1-t)]^{\epsilon_{f_0,\pi_0}}.
\end{equation}
Suppose that $\epsilon_{f_0,\pi_0}=1$. By \eqref{(50)} we get
\begin{equation}\label{(53)}
(1-\alpha_{f_0,\pi_0})f_0(1-t)=(1+\alpha_{f_0,\pi_0})t-\alpha_{f_0,\pi_0},\qquad t\in [0,1].
\end{equation}
Then we have $0=1-2t$ for every $t\in [0,1]$ if $\alpha_{f_0,\pi_0}=1$, which is impossible, so that $\alpha_{f_0,\pi_0}\ne 1$. Then by \eqref{(53)} we get
\[
f_0(1-t)=\frac{1+\alpha_{f_0,\pi_0}}{1-\alpha_{f_0,\pi_0}}t-\frac{\alpha_{f_0,\pi_0}}{1-\alpha_{f_0,\pi_0}},\qquad t\in [0,1],
\]
so that
\[
f_0(t)=\frac{1+\alpha_{f_0,\pi_0}}{1-\alpha_{f_0,\pi_0}}(1-t)-\frac{\alpha_{f_0,\pi_0}}{1-\alpha_{f_0,\pi_0}},\qquad t\in [0,1],
\]
which is a contradiction to $f_0\in W_1$. We have that $\epsilon_{f_0,\pi_0}\ne 1$, hence $\epsilon_{f_0,\pi_0}=-1$. Then by \eqref{(50)} we get
\begin{equation}\label{(60)}
f_0(1-t)=(1+\alpha_{f_0,\pi_0})t-\alpha_{f_0,\pi_0}+\alpha_{f_0,\pi_0}\overline{f_0(1-t)}, \qquad t\in [0,1].
\end{equation}
Thus 
\begin{multline}\label{(61)}
f_0(1-t)=(1+\alpha_{f_0,\pi_0})t-\alpha_{f_0,\pi_0}\\
+\alpha_{f_0,\pi_0}\overline{((1+\alpha_{f_0,\pi_0})t-\alpha_{f_0,\pi_0}+\alpha_{f_0,\pi_0}\overline{f_0(1-t)})},\quad t\in [0,1].
\end{multline}
As $|\alpha_{f_0,\pi_0}|=1$ we get
\begin{multline}\label{(62)}
f_0(1-t)=(1+\alpha_{f_0,\pi_0})t-\alpha_{f_0,\pi_0}\\
+\alpha_{f_0,\pi_0}\overline{(1+\alpha_{f_0,\pi_0})}t-1+f_0(1-t),\quad t\in [0,1].
\end{multline}
Hence 
\begin{equation}\label{(63)}
0=\left((1+\alpha_{f_0,\pi_0})+\alpha_{f_0,\pi_0}(1+\overline{\alpha_{f_0,\pi_0}})\right)t-(\alpha_{f_0,\pi_0}+1),\qquad t\in [0,1].
\end{equation}
Thus $\alpha_{f_0,\pi_0}=-1$. Substituting $\alpha_{f_0,\pi_0}=-1$ into \eqref{(50)} we get
\[
f_0(1-t)=1-\overline{f_0(1-t)},\qquad t\in [0,1]
\]
since $\epsilon_{f_0,\pi_0}=-1$. Then 
\[
f_0([0,1])=1-\overline{f_0([0,1])},
\]
which contradicts to $f_0\in W_1$. It follows that \eqref{(13')} does not occur.

Assuming the existence of $f_0\in W_1$ such that $T_1(f_0)\ne f_0$ we arrived at the contradiction. We conclude that $T_1(f)=f$ for every $f\in W_1$. 

Let $g\in A_1$ Then by Proposition \ref{PsubsetBarU} there is a sequence $\{g_n\}$ in $W_1$ such that $\|g-g_n\|_{\infty}\to 0$ as $n\to \infty$. By the previous part of the proof we have
\begin{equation}\label{(70)}
T_1(g_n)=g_n
\end{equation}
for every $n$. By Lemma \ref{lemma0}, $T_1$ is an symmetry with respect to $\|\cdot\|_{\infty}$, we have $T_1(g)=g$ by letting $n\to \infty$ for \eqref{(70)}. We conclude that $T_1(g)=g$ for every $g\in A_1$ if $T_1(\pi_0)=\pi_0$. It follows that 
\[
T(g)=T(0)+\alpha[g]^{\epsilon},\qquad g\in A_1
\]
if $T_1(\pi_0)=\pi_0$. Suppose that $T_1(\pi_0)=\pi_1$. As we have already described, we see that $T_1(g)=g\circ\pi_1$ for every $g\in A_1$. Hence we have
\[
T(g)=T(0)+\alpha[g\circ\pi_1]^\epsilon,\qquad g\in A_1.
\]
Thus we observed that $T\in G_{\Pi_0}$. As we have already proved that $T_1$ is a surjective isometry from $A_1$ onto $A_2$, we see that $T$ is also a surjective isometry. Hence we conclude that $T\in \gpa$.
\end{proof}
\section{Surjective real-linear isometries on $\lip(K)$}\label{sec4}
Jarosz and Pathak  had exhibited in \cite[Example 8]{jp} the form of a surjective {\it complex-linear} isometry on the Banach algebra $\lip(K_j)$ with the norm $\|\cdot\|_{\Sigma}$ of the Lipschitz functions on a compact metric space $K_j$ by answering the question posed by Rao and Roy \cite{rr}.  After the publication of \cite{jp} some authors expressed their suspicion about the argument there and the validity of the statement there had not been confirmed until the correction \cite[Corollary 15]{szeged} was published by Hatori and Oi. In this section by applying \cite[Lemmas 10,11]{szeged} and \cite[Proposition 7]{spanish} we exhibit the form of a surjective real-linear isometry between the Banach algebras of Lipschitz functions. 
\begin{lemma}\label{13pegemae}
Let $p,q\in \mc$. Suppose that $|p+\lambda q|=1$ for at least three different unimodular $\lambda\in \mc$. Then $p=0$ and $|q|=1$, or $|p|=1$ and $q=0$.
\end{lemma}
\begin{proof}
Let $\lambda_1,\lambda_2,\lambda_3$ be three unimodular complex numbers such that $|p+\lambda_j q|=1$ for $j=1,2,3$. Suppose that $q\ne0$. Then $|p/q+\lambda_j|=1/|q|$, by which the circle of the center $-p/q$ and the radius $1/|q|$ is the outer tangent circle of the triangle defined by the three point $\lambda_1,\lambda_2,\lambda_3$. On the other hand, the unit circle is the outer tangent circle of  these three points since $|\lambda_j|=1$ for $j=1,2,3$. By the uniqueness of the outer tangent circle, we see that $|q|=1$ and $p/q=0$, thus $p=0$. On the other hand, if $q=0$, then it is apparent that $|p|=1$
\end{proof}
\begin{theorem}\label{lip}
Let $K_j$ be a compact metric space for $j=1,2$. Suppose that $U:\lip(K_1)\to \lip(K_2)$ is a surjective real-linear isometry with respect to the norm $\|f\|_{\Sigma}=\|f\|_{\infty}+L_f$ for $f\in \lip(K_1)$. Then there exists a surjective isometry $\pi:K_2\to K_1$ such that 
\[
U(f)=U(1)f\circ \pi,\qquad f\in \lip (K_1)
\]
or 
\[
U(f)=U(1)\overline{f\circ\pi},\qquad f\in \lip(K_1).
\]
\end{theorem}
\begin{proof}
As the same way as in the proof of \cite[Proposition 9]{szeged} we apply Choquet's theory. Let $j=1,2$.  Let $\mf_j$ be the Stone-\v Cech compactification of $\{(x,x')\in K_j^2:x\ne x'\}$. For $f\in \lip(K_j)$, let $D_j(f)$ denote the continuous extension to $\mf_j$ of the function $\left(f(x)-f(x')\right)/d(x,x')$ on $\{(x,x')\in K_j^2:x\ne x'\}$. Then $D_j:\lip(K_j)\to C(\mf_j)$ is well defined. We have $\|D_j(f)\|_{\infty}=L_f$ for every $f\in \lip(K_j)$. Then $(K_j,\mc, \lip(K_j),\lip(K_j))$ is an admissible quadruple of type L (see \cite[Definition 4, Example 12]{szeged}). 
Let  ${\mathbb T}=\{\z\in \mc:|z|=1\}$. 
For $j=1,2$, 
define a map 
\[
I_j:\lip(K_j)\to C(K_2\times \mf_2\times {\mathbb T})
\]
by $I_j(f)(x,m,\gamma)=f(x)+\gamma D_j(f)(m)$ for $f\in \lip(K_j)$ and $(x,m,\gamma)\in K_j\times \mf_j\times {\mathbb T}$. As $D_j$ is a complex-linear map, so is $I_j$. For simplicity we write $\tilde{f}=I_j(f)$ for $f\in \lip(K_j)$. For every $f\in \lip(K_j)$ we have
\[
\|\tilde{f}\|_{\infty}=\|f\|_{\infty}+\|D(f)\|_{\infty}=\|f\|_{\infty}+L_f,
\]
hence $I_j$ is a surjective complex-linear isometry from $\lip(K_j)$ onto $B_j=I_j(\lip(K_j))$. 
We have $D(1)=0$ and $\tilde{c}=c$ for every $c\in \mc$, where the constant function taking the value $c$ is denoted also by $c$. It follows that $B_j$ is a complex-linear closed subspace of $C(K_j\times \mf_j\times{\mathbb T})$ which contains $1$. A point $p=(x,m,{\mathbb T})\in K_2\times \mf_j\times {\mathbb T}$ is in the Choquet boundary $\ch B_j$ if the point evaluation $\delta_p$  is an extreme point of the closed unit ball $B_{j,1}^*$ of the (complex) dual space $B_j^*$ of $B_j$. See the description just after \cite[Proposition 9]{szeged}.
 Define $\SSS:B_1\to B_2$ by $\SSS(\tilde{f})=\widetilde{U(f)}$ for $\tilde{f}\in B_1$. Then $S$ is a surjective real-linear isometry from $B_1$ onto $B_2$. 

Let $x_0\in K_2$. 
Put $b_0(x)=1-\frac{d(x,x_0)}{d(K_2)}$ for $x\in K_2$, where $d(K_2)$ is the diameter of $K_2$. By a simple calculation we have $b_0\in \lip(K_2)$, $0\le b_0\le 1$ on $K_2$, and $b_0(x)=1$ if and only if $x=x_0$. Then by \cite[Lemma 10]{szeged} there exists a pair $(m_0,\gamma_0)\in \mf_2\times{\mathbb T}$ such that $(x_0,m_0,\gamma_0)\in \ch B_2$. By \cite[Lemma 11]{szeged} we have that $p_{\theta}=(x_0,m_0,e^{i\theta}\gamma_0)\in \ch B_2$ for every $0<\theta<\pi/2$.  For $\eta\in B_2^*$ we define $S_*(\eta)\in B_1^*$ by
\[
(\SSS_*(\eta))(\tilde{f})=\Real \eta(\SSS(\tilde{f}))-i\Real \eta(\SSS(i\tilde{f})),\qquad f\in B_1.
\]
Then $\SSS_*:B^*_2\to B_1^*$ is a surjective complex-linear isometry (cf. \cite[(2.3)]{mt}). Denote the set of all extreme points in $\{\nu\in B_j^*:\|\nu\|\le 1\}$ by $\ext B_j^*$. As $S_*$ is a surjective complex-linear isometry, we have that $\eta\in \ext B_2^*$ if and only if $\SSS_*(\eta)\in \ext B_1^*$. By the definition of the Choquet boundary, the point evaluation $\delta_{p_\theta}:B_2\to \mc$ defined by $\delta_{p_\theta}(\tilde{f})=\tilde{f}(p_{\theta})$, $f\in B_2$ is in $\ext B_2^*$. Then the Arens-Kelley theorem asserts that
\[
\SSS_*(\delta_{p_{\theta}})=\lambda_1 \delta_{p_1}
\]
for a unimodular $\lambda_1\in \mc$ and a $p_1\in \ch B_1$. In the same way there exist a unimodular $\lambda_i\in \mc$ and $p_i\in \ch B_1$ such that 
\[
\SSS_*(i\delta_{p_\theta})=\lambda_i \delta_{p_i}.
\]
By the same way as in the proof of \cite[Lemma 3.3]{mt} we have that $\lambda_i=i\lambda_1$ or $-i\lambda_1$. For a convenient of the readers we show a proof. As $p_{\theta}\in \ch B_2$, $\frac{1+i}{\sqrt{2}}\delta_{p_\theta}$ is in $\ext B_2^*$. Then there exists a unimodular $\mu\in \mc$ and a $q\in \ch B_1$ such that 
\[
\frac{1}{\sqrt{2}}\left(\lambda_1\delta_{p_1}+\lambda_i\delta_{p_i}\right)=S_*\left(\frac{1+i}{\sqrt{2}}\delta_{p_\theta}\right)=\mu\delta_q.
\]
Substituting $1\in B_1$ in this equation to get 
\[
\frac{1}{\sqrt{2}}(\lambda_1+\lambda_i)=\mu.
\]
As $|\mu|=1$ we have 
\[
|\lambda_1+\lambda_i|=\sqrt{2}.
\]
As $|\lambda_1|=|\lambda_i|=1$ we conclude that $\lambda_i=i\lambda_1$ or $\lambda_i=-i\lambda_1$. Put $\varepsilon_0=\lambda_i/\lambda_1$.

Let $c\in \mc$ be arbitrary. For simplicity we also write $c$ as the constant function which takes the value $c$. We have 
\begin{eqnarray}\label{hen1}
\left(\SSS_*(\delta_{p_{\theta}})\right)(\tilde{c})&=&\Real \delta_{p_{\theta}}(\SSS(\tilde{c}))-i\Real \delta_{p_\theta}(S(i\tilde{c})) \\
&=&\Real(\SSS(\tilde{c}))(p_{\theta})-i\Real (\SSS(i\tilde{c}))(p_{\theta}),\nonumber
\end{eqnarray}
and
\begin{eqnarray}\label{hen2}
\left(\SSS_*(i\delta_{p_\theta})\right)(\tilde{c})&=&\Real i\delta_{p_{\theta}}(\SSS(\tilde{c}))-i \Real i\delta_{p_\theta}(\SSS(i\tilde{c})) \\
&=&-\Imarg(\SSS(\tilde{c}))(p_{\theta})+i\Imarg(\SSS(i\tilde{c}))(p_{\theta}).\nonumber
\end{eqnarray}
As $\SSS_*(\delta_{p_{\theta}})=\lambda_1\delta_{p_1}$, we have by \eqref{hen1} that
\[
\Real (\SSS(\tilde{c}))(p_{\theta})=\Real \left(\SSS_*(\delta_{p_\theta})\right)(\tilde{c})=\Real \lambda_1(\tilde{c})(p_1)=\Real \lambda_1c.
\]
As $\SSS_*(i\delta_{p_\theta})=i\varepsilon_0\lambda_1\delta_{p_i}$, we have by  \eqref{hen2} that 
\[
\Imarg (\SSS(\tilde{c}))(p_\theta)=-\Real \left(\SSS_*(i\delta_{p_{\theta}})\right)(\tilde{c})=-\Real\left(i\varepsilon_0\lambda_1(\tilde{c})(p_i)\right)=\Imarg \varepsilon_0 \lambda_1c.
\]
Thus
\begin{equation}\label{280}
(\SSS(\tilde{c}))(p_{\theta})=\Real \lambda_1c+i\Imarg \varepsilon_0 \lambda_1c=
\begin{cases}
\lambda_1c,\quad \text{if}\quad \varepsilon_0=1, \\
\overline{\lambda_1c} \quad \text{if}\quad \varepsilon_0=-1
\end{cases}
\end{equation}
On the other hand we have by the definition of $\SSS$ that
\begin{equation}\label{290}
\SSS(\tilde{c})(p_{\theta})=(U(c))(x_0)+e^{i\theta}\lambda_0(D_2(U(c)))(m_0).
\end{equation}
Combining \eqref{280} and \eqref{290} we have
\begin{equation}\label{300}
\left|(U(c))(x_0)+e^{i\theta}\lambda_0(D_2(U(c)))(m_0)\right|=|c|
\end{equation}
for every $c\in \mc$. Substituting $c=1$ in \eqref{300} we get
\begin{equation}\label{301}
\left|(U(1))(x_0)+e^{i\theta}\lambda_0(D_2(U(1)))(m_0)\right|=1
\end{equation}
for $0<\theta <\pi/2$. 
By Lemma \ref{13pegemae} we have $(U(1))(x_0)=0$ or $(D_2(U(1)))(m_0)=0$. But $(U(1))(x_0)=0$ is impossible. The reason is as follows. Suppose that $(U(1))(x_0)=0$. Then $\left|(D_2(U(1)))(m_0)\right|=1$ by \eqref{301}. Hence $\|D_2(U(1))\|_{\infty}\ge1$. Since $U$ is an isometry we get
\[
1=\|U(1)\|_{\Sigma}=\|U(1)\|_{\infty}+\|D_2(U(1))\|_{\infty}\ge 1,
\]
hence $\|U(1)\|_{\infty}=0$, therefore $U(1)=0$ on $K_2$ and $D_2(U(1))=0$ follows, which is against to $\|D(U(1))\|_{\infty}\ge 1$. Thus we have that $(D_2(U(1)))(m_0)=0$, so that $|(U(1))(x_0)|=1$. By
\begin{multline*}
1\le|(U(1))(x_0)|+\|D_2(U(1))\|_{\infty} \\
\le\|U(1)\|_{\infty}+\|D_2(U(1))\|_{\infty}=\|U(1)\|_{\infty}+L_{U(1)}=1,
\end{multline*}
we conclude that $\|D_2(U(1))\|_{\infty}=0$, so $D_2(U(1))=0$, thus $U(1)$ is a constant function by the definition of $D_2$. As $|(U(1))(x_0)|=1$, we infer that $U(1)$ is a constant function of unit modulus. In the same way, substituting $c=i$ in \eqref{300} we have that $U(i)$ is a constant function of unit modulus. Since $U(1)-U(i)$ is a constant function we have
\[
\sqrt{2}=|1-i|=\|U(1)-U(i)\|_{\Sigma}=\|U(1)-U(i)\|_{\infty}
\]
as $0=L_{U(1)-U(i)}=D_2(U(1)-U(i))$. Since $U(1)$ and $U(i)$ are constant functions we infer that $U(i)=iU(1)$ or $U(i)=-iU(1)$. Put $U_0=\overline{U(1)}U$. Then $U_0$ is a surjective real-linear isometry from $\lip(K_1)$ onto $\lip(K_2)$ such that $U_0(1)=1$ and $U_0(i)=i$ or $-i$. Applying \cite[Proposition 7]{spanish} we see that $U_0$ is also an isometry with respect to the supremum norm $\|\cdot\|_{\infty}$ on $K_1$ and $K_2$ respectively, hence $U_0$ is extended to a surjective isometry $\widetilde{U_0}$ between the uniform closure of $\lip(K_j)$, which coincides with $C(K_j)$ by the Stone-Weierstrass theorem. Thus 
\[
\widetilde{U_0}:C(K_1)\to C(K_2)
\]
is a surjective real-linear isometry with respect to the supremum norm. As $K_j$ is the \v Silov boundary for $C(K_j)$ we can apply \cite[Theorem]{ellis} to have that there exists a homeomorphism $\pi':K_2\to K_1$ and an open and closed subset $E_2$ of $K_2$ such that 
\begin{equation*}
\widetilde{U_0(f)}= 
\begin{cases}
f\circ\pi',\quad \text{on}\quad E_2,\\
\overline{f\circ\pi'}, \quad \text{on}\quad K_2\setminus E_2
\end{cases}
\end{equation*}
for $f\in C(K_1)$ (cf. \cite{hm,m}). As $\widetilde{U_0}(i)=U_0(i)=i$ or $-i$ we have that
\[
\widetilde{U_0}(f)=f\circ \pi',\qquad f\in C(K_1)
\]
if $U_0(i)=i$ and 
\[
\widetilde{U_0}(f)= \overline{f\circ\pi'},\qquad f\in C(K_1)
\]
if $U_0(i)=-i$. It follows that $U_0$ is a complex-linear map if $U_0(i)=i$ and $\overline{U_0}$ is a complex-linear map if $U_0(i)=-i$. Applying \cite[Corollary 15]{szeged} there exists a surjective isometry $\pi:K_1\to K_2$ such that 
\[
U_0(f)=f\circ \pi,\qquad f\in \lip(K_1)
\]
if $U_0(i)=i$ and 
\[
\overline{U_0}(f)=f\circ\pi,\qquad f\in \lip(K_1)
\]
if $U_0(i)=-i$. It follows that 
\[
U(f)=U(1)f\circ\pi,\qquad f\in \lip(K_1)
\]
or 
\[
U(f)=U(1)\overline{f\circ\pi},\qquad f\in \lip(K_1).
\]
\end{proof}
Let $\Pi$ be the set of all surjective isometries from $K_2$ onto $K_1$. 
\begin{cor}\label{lipiso}
$G_{\Pi}(\operatorname{Lip}(K_1),\lip(K_2))=\operatorname{Iso}(\operatorname{Lip}(K_1),\lip(K_2))$
\end{cor}
\begin{proof}
Suppose that $T\in \operatorname{Iso}(\operatorname{Lip}(K_1),\lip(K_2))$. By the Mazur-Ulam theorem we have that $U=T-T(0)$ is a surjective real-linear isometry from $\lip(K_1)$ onto $\lip(K_2)$. By Theorem \ref{lip} $U(1)$ is a constant function of unit modulus and there exists a surjective isometry $\pi\in \Pi$ such that $U(f)=U(1)f\circ\pi$ for every $f\in \lip(K_1)$ or $U(f)=U(1)\overline{f\circ\pi}$ for every $f\in \lip(K_1)$. It follows that $T=U+T(0)\in G_{\Pi}(\operatorname{Lip}(K_1),\lip(K_2))$.

Suppose that $T\in G_{\Pi}(\operatorname{Lip}(K_1),\lip(K_2))$. It is a routine work to show that $T$ is a surjective real-linear isometry.
\end{proof}
\section{Applications}\label{sec5}
In this section we study the problem on 2-locality for $C^1[0,1]$ and $\lip[0,1]$. We first prove that $\isonly(\lip[0,1])$ with the norm $\|\cdot\|_{\Sigma}$ is 2-local reflexive in $M(\lip[0,1])$. The Banach algebra $\lip[0,1]$   satisfies the three conditions 1), 2) and 3) for $A_j$ in the first part of section \ref{sec3}. 

Recall that $\pi_0$ is the identity map on the interval $[0,1]$, $\pi_1=1-\pi_0$ and $\Pi_0=\{\pi_0,\pi_1\}$.
\begin{theorem}
$\isonly(\lip[0,1])$ is 2-local reflexive in $M(\lip[0,1])$, where $\lip[0,1]$ is the Banach algebra of all Lipschitz functions defined on the closed interval $[0,1]$ with the norm $\|\cdot\|_{\Sigma}$.
\end{theorem}
\begin{proof}
Corollary \ref{lipiso} asserts that $G_{\Pi_0}(\lip[0,1])=\isonly(\lip[0,1])$, hence $\isonly(\lip[0,1])=G_{\Pi_0}\cap\isonly(\lip[0,1])$. The Banach algebra $\lip[0,1]$ satisfies the conditions 1), 2) and 3) for $A_j$ in the first part of section \ref{sec3}. Applying Theorem \ref{2-localgeneral} for $A_j=\lip[0,1]$, we infer that $\isonly(\lip[0,1])$ is 2-local reflexive in $M(\lip[0,1])$.
\end{proof}
Next we prove that $\isonly(C^1[0,1])$ is 2-local reflexive in $M(C^1[0,1])$ for certain norms. Let $D$ be a non-empty connected compact subset of $[0,1]\times [0,1]$. The norm $\|\cdot\|_{\langle D\rangle}$ on $C^1[0,1]$ is defined by 
\[
\|f\|_{\langle D\rangle}=\sup_{(r,s)\in D}(|f(r)|+|f'(s)|),\quad f\in C^1[0,1].
\]
Let $P_j:[0,1]\times[0,1] \to [0,1]$ be the projection onto the $j$-th factor ($j=1,2$). 
Let $D$ be a non-empty connected compact subset of $[0,1]\times [0,1]$. The norm $\|\cdot\|_{\langle D\rangle}$ on $C^1[0,1]$ is defined by Kawamura, Koshimizu and Miura \cite{kkm} :
\[
\|f\|_{\langle D\rangle}=\sup_{(r,s)\in D}(|f(r)|+|f'(s)|),\quad f\in C^1[0,1].
\]
They study surjective real-linear isometries between $C^1[0,1]$ onto itself for the norm $\|\cdot\|_{\langle D\rangle}$ under additional hypothesis on $D$. The main result of \cite{kkm} exhibits the form of isometries on $C^1[0,1]$ for a wide class of norms and unifies the former results on isometries for several important norms such as $\|\cdot\|_{\Sigma}$, $\|\cdot\|_{\sigma}$, $\|\cdot\|_{\Delta}$ and so on. 
If $D=[0,1]\times [0,1]$, then $\|f\|_{\langle D\rangle}=\|f\|_{\infty}+\|f'\|_{\infty}$ for $f\in C^1[0,1]$. If $D=\{(t,t):t\in [0,1]\}$, then $\|f\|_{\langle D\rangle}=\sup\{|f(t)|+|f'(t)|:t\in [0,1]\}$. We point out that applying the Mazur-Ulam theorem, their results in fact assures the forms of surjective isometries without the assumption of linearity. 
\begin{theorem}
Let $D$ be a non-empty connected compact subset of $[0,1]\times [0,1]$. Suppose that $P_j(D)=[0,1]$ for $j=1,2$. Then 
$\operatorname{Iso}(C^1[0,1],\|\cdot\|_{\langle D\rangle})$ is 2-local reflexive in $M(C^1[0,1])$.
\end{theorem}
\begin{proof}
By \cite[Corollary]{kkm} and the Mazur-Ulam theorem we infer that $G_{\Pi_0}=\isonly(C^1[0,1])$. The Banach space $C^1[0,1]$ satisfies the conditions 1), 2) and 3) for $A_j$ in the first part of section \ref{sec3}. Applying Theorem \ref{2-localgeneral} for $A_j=C^1[0,1]$, we infer that $\isonly(C^1[0,1])$ is 2-local reflexive in $M(\lip[0,1])$.
\end{proof}

We point out that the 2-locality problem for surjective isometries without assuming linearity is much difficult than the 2-locality problem for surjective complex or even real linear isometries.
We do not know if $\isonly(C[0,1],\|\cdot\|_{\infty})$ is 2-local reflexive or not.

\subsection*{Acknowledgments}
The first author was supported by JSPS KAKENHI Grant Numbers JP16K05172, JP15K04921. He would like to express his heartful thanks to Lajos Moln\'ar for his hospitality during staying in the University of Szeged on Octorber 2018 and for mentioning the problem of 2-local reflexivity of the groups of all surjective isometries. 

\end{document}